\documentclass[reqno,10pt]{amsart}
\usepackage{graphicx}
\usepackage{amsmath,amsthm,amsfonts,amssymb,latexsym,epsfig}
\newtheorem{theorem}{Theorem}[section]
\newtheorem{lemma}[theorem]{Lemma}

\theoremstyle{definition}

\usepackage{hyperref}
\theoremstyle{remark}

\usepackage{orcidlink}
\usepackage[a4paper, left=1in, right=1in, top=1in, bottom=1in]{geometry}
\numberwithin{equation}{section}

\begin{document}

\title[short text for running head]{``On $\ell-$regular and $2-$color partition triples''}
\title{On $\ell-$regular and $2-$color partition triples modulo powers of $3$}

\author[B. Hemanthkumar and D. S. Gireesh]{B. Hemanthkumar $^{1}$\orcidlink{0000-0001-7904-293X}}

\address{$^1$Department of Mathematics, RV College of Engineering, RV Vidyanikethan Post, Mysore Road, Bengaluru-560 059, Karnataka, India.}
\email{hemanthkumarb.30@gmail.com}

\thanks{}

\author[]{D. S. Gireesh$^{2}$\orcidlink{0000-0002-2804-6479}}
\address{$^{2}$Department of Mathematics, BMS College of Engineering, P.O. Box No.: 1908, Bull Temple Road,
Bengaluru-560 019, Karnataka, India.}
\email{gireeshdap@gmail.com}
\thanks{}

\date{}
\medskip
\subjclass[2020]{primary 11P83; secondary 05A15, 05A17}
\keywords{Partitions; $\ell$-Regular Partitions; Generating Functions; Congruences}

\begin{abstract}
Let $T_\ell(n)$ denote the number of $\ell-$regular partition triples of $n$ and let $p_{\ell, 3}(n)$ enumerates the number of 2--color partition triples of $n$ where one of the colors appear only in parts that are multiples of $\ell$. In this paper, we prove several infinite families of congruences modulo powers of 3 for $T_\ell(n)$ and $p_{\ell, 3}(n)$, where $\ell \geq 1$ and $\equiv 0\pmod{3^k}$, and $\equiv \pm 3^k \pmod{3^{k+1}}$.
\end{abstract}
\maketitle
\section{Introduction}
Let $p_3(n)$ denotes the number of partitions of $n$ in three colors, given by
\begin{equation}\label{D1}
\sum_{n\geq 0}p_3(n)q^n=\frac1{E(q)^3}.
\end{equation}
Here and throughout the paper, we set
\[E(q^r):=(q^r;q^r)_\infty=\prod\limits_{m=1}^{\infty}(1-q^{rm}).\]
In \cite{H}, Hirschhorn showed that for $\alpha\geq0,$
\begin{equation}\label{H1}
\sum_{n\geq 0}p_3\left(3^{2\alpha+1}n+\frac{5\times3^{2\alpha+1}+1}{8}\right)q^n=\sum_{j\geq 1}x_{2\alpha+1,j}q^{j-1}\frac{E(q^3)^{12j-3}}{E(q)^{12j}}
\end{equation}
and
\begin{equation}\label{H2}
\sum_{n\geq 0}p_3\left(3^{2\alpha+2}n+\frac{7\times3^{2\alpha+2}+1}{8}\right)q^n=\sum_{j\geq 1}x_{2\alpha+2,j}q^{j-1}\frac{E(q^3)^{12j}}{E(q)^{12j+3}}.
\end{equation}
where the coefficient vectors $\mathbf{x}_k= (x_{k,1}, x_{k,2},\dots)$ are given by
$$\mathbf{x}_1=\left(x_{1,1},x_{1,2},\dots\right)=\left(9,0,0,\dots\right),$$ and for $k\geq1$,
\begin{equation*}
   x_{k+1, i}=\begin{cases}
  \displaystyle \sum_{j\geq1} x_{k, j} \,m_{4i,i+j}, &\text{if $k$ is odd},\\
   \displaystyle \sum_{j\geq1}x_{k, j} \,m_{4i+1,i+j}, &\text{if $k$ is even},
    \end{cases}
\end{equation*}
where the first five rows of $M=\left(m_{i,j}\right)_{i,j\geq 1}$ are
\[\begin{bmatrix}
3^2 & 0 & 0 & 0& 0& 0&\dots\\
2\times 3 & 3^5 & 0 &0 &0 & 0& \dots\\
1 & 3^5 & 3^8 &0 &0 & 0& \dots\\
0 & 2\times3^2\times 5 & 2^2\times3^7 &3^{11} &0 & 0& \dots\\
0 & 3\times5 & 2^2\times3^5\times 5 &3^{10}\times 5 & 3^{14} & 0& \dots
\end{bmatrix}\]
and for $i>3$, $m_{i,1}=0$, and for $j\geq2$,
\begin{equation}\label{R1}
m_{i,j}=27m_{i-1,j-1}+9m_{i-2,j-1}+m_{i-3,j-1}.
\end{equation}

For any positive integer $\ell$, a partition is called $\ell$-regular if none of its parts are divisible by $\ell$. Let $T_{\ell}(n)$ enumerate the number of $\ell$-regular partition triples of $n$. It is given by
\begin{equation}\label{D2}
\sum_{n\geq 0}T_{\ell}(n)q^n=\frac{E(q^\ell)^3}{E(q)^3}.
\end{equation}
Recently, Gireesh and Naika \cite{GN} found infinite family of congruences for $T_3(n)$. They showed that for $\beta \geq 0$ and $n\geq 0$
\begin{equation}\label{G1}
T_3\left(3^{2\beta+1}n+\frac{3^{2\beta+2}-1}{4}\right) \equiv 0 \pmod{3^{2\beta+2}}.
\end{equation}
Baruah and Das \cite{NH} extented these results to 9 and 27--regular partition triples. For each $\beta, n \geq 0$ and $r\in\{7, 11\},$ they proved that
\begin{align}
T_9(3^{\beta+1}(n+1)-1) &\equiv 0 \pmod{3^{2\beta+2}}, \label{B1}\\
T_9(3^{\beta+2}n+2\cdot 3^{\beta+1}-1) &\equiv  0 \pmod{3^{2\beta+3}}, \label{B2} \\
T_{27}\left(3^{2\beta+3}n+\frac{3^{2\beta+4}-13}{4}\right) &\equiv 0 \pmod{3^{2\beta+5}}, \label{B3}
\end{align}
and
\begin{equation}\label{B4}
T_{27}\left(3^{2\beta+4}n+\frac{r\cdot3^{2\beta+3}-13}{4}\right) \equiv 0 \pmod{3^{2\beta+7}}.
\end{equation}
Further, they conjectured that for all $\ell, n \geq 0$ and $k \geq 1$,
\begin{equation}\label{BC1}
T_{3^{2k}}\left(3^{\ell+2k-1}n+3^{\ell+2k-1}-\frac{3^{2k}-1}{8}\right) \equiv 0 \pmod{3^{3k+2\ell-1}}
\end{equation}
and
\begin{equation}\label{BC2}
T_{3^{2k-1}}\left(3^{2\ell+2k-1}n+\frac{2\cdot3^{2\ell+2k}-3^{2k-1}+1}{8}\right) \equiv 0 \pmod{3^{3k+2\ell-1}}.
\end{equation}
The $2-$color partition triple function is defined by 
\begin{equation}\label{PD3}
\sum_{n\geq 0}p_{\ell, 3}(n)q^n=\frac{1}{E(q)^3 E(q^\ell)^3}.
\end{equation}
In \cite{T}, Tang proved congruences modulo powers of 3 for 2--color partition triples. In particular, he proved that
\begin{align}
p_{3, 3} \left(3^{\beta+1}+\frac{3^{\beta+1}+1}{2} \right) &\equiv 0 \pmod{3^{\beta+2}},\label{T1} \\
p_{9, 3} \left(3^{2\beta+1}n+\frac{3^{2\beta+1}+5}{4} \right) &\equiv 0 \pmod{3^{2\beta+2}}, \label{T2} 
\end{align}
and
\begin{equation} \label{T3}
p_{9, 3} \left(3^{2\beta+2}n+\frac{3^{2\beta+3}+5}{4} \right) \equiv 0 \pmod{3^{2\beta+4}}.
\end{equation}

In this paper, we investigate the arithmetic properties of partition functions $T_\ell(n)$ and $p_{\ell, 3}(n)$, where $\ell$ is a positive integer congruent to $0\pmod{3^k}$, and $\pm 3^k \pmod{3^{k+1}}$. We establish several infinite families of congruences modulo powers of 3 and confirm conjectures \eqref{BC1} and \eqref{BC2}.

The main results are as follows:
\begin{theorem}\label{MT1} For any integers $n, \alpha, \beta, k \geq 0$, and prime $p$ such that $p \equiv 3 \pmod{4}$ and $p\nmid n$, we have
\begin{align}
T_{3^{2\alpha+1}}\left(3^{2\alpha+2\beta+1}n+\frac{2\times3^{2\alpha+2\beta+2}-3^{2\alpha+1}+1}{8}\right ) & \equiv 0\pmod{3^{3\alpha+2\beta+2}}, \label{MR1} \\
T_{3^{2\alpha+1}}\left(3^{2\alpha+2\beta+2}n+\frac{14\times3^{2\alpha+2\beta+1}-3^{2\alpha+1}+1}{8}\right ) & \equiv 0\pmod{3^{3\alpha+2\beta+4}}, \label{MR2} \\
T_{3^{2\alpha+1}}\left(3^{2\alpha+2\beta+2}n+\frac{22\times3^{2\alpha+2\beta+1}-3^{2\alpha+1}+1}{8}\right ) & \equiv 0\pmod{3^{3\alpha+2\beta+5}}, \label{MR3}
\end{align}
and
\begin{equation}
T_{3^{2\alpha+1}}\left(3^{2\alpha+2\beta+2}p^{2k+1}n+\frac{2p^{2k+2}\times3^{2\alpha+2\beta+2}-3^{2\alpha+1}+1}{8}\right ) \equiv 0\pmod{3^{3\alpha+2\beta+4}}. \label{MR4} 
\end{equation}
\end{theorem}

\begin{theorem}\label{MT2} For any integers $n, \alpha, \beta \geq 0$, we have
\begin{equation}
T_{3^{2\alpha+2}}\left(3^{2\alpha+\beta+1}(n+1)-\frac{3^{2\alpha+2}-1}{8}\right )  \equiv 0\pmod{3^{3\alpha+2\beta+2}} \label{MR5} 
\end{equation}
and
\begin{equation}
T_{3^{2\alpha+2}}\left(3^{2\alpha+\beta+1}(3n+2)-\frac{3^{2\alpha+2}-1}{8}\right )  \equiv 0\pmod{3^{3\alpha+2\beta+3}}. \label{MR6}  
\end{equation}
\end{theorem}
For any positive integer $\ell$ congruent to $\pm 3^k \pmod{3^{k+1}}$, we denote $T_\ell(n)$ and $p_{\ell, 3}(n)$ by $T_{\pm3^{k}( 3^{k+1})}(n)$ and $p_{\pm3^{k}( 3^{k+1}), 3}(n)$, respectively.
\begin{theorem}\label{MT3} For any integers $n, \alpha, \beta \geq 0$, we have
\begin{align}
T_{\pm3^{2\alpha+1}( 3^{2\alpha+2})}\left(3^{2\alpha+2\beta+1}n+\frac{3^{2\alpha+2\beta+2}+2\times3^{2\alpha+1}+1}{8}\right)  & \equiv 0 \pmod{3^{3\alpha+\beta+2}},\label{MR7} \\
T_{\pm3^{2\alpha+1}( 3^{2\alpha+2})}\left(3^{2\alpha+2\beta+2}n+\frac{11\times3^{2\alpha+2\beta+1}+2\times3^{2\alpha+1}+1}{8}\right)  & \equiv 0 \pmod{3^{3\alpha+\beta+4}},\label{MR8}\\
T_{\pm3^{2\alpha+1}( 3^{2\alpha+2})}\left(3^{2\alpha+2\beta+2}n+\frac{19\times3^{2\alpha+2\beta+1}+2\times3^{2\alpha+1}+1}{8}\right) &\equiv 0 \pmod{3^{3\alpha+\beta+5}},\label{MR9}\\
\nonumber T_{\pm3^{2\alpha+1}( 3^{2\alpha+2})}\left(3^{2\alpha+2\beta+2}n+\frac{3^{2\alpha+2\beta+2}+2\times3^{2\alpha+1}+1}{8}\right)\\\equiv \begin{cases} (-1)^n (2n+1) \pmod{3^{3\alpha+\beta+4}} & \text{if $n=k(k+1)/2$},\\
0\pmod{3^{3\alpha+\beta+4}} &\text{otherwise},\end{cases} \label{MR10}
\end{align}
and
\begin{align}
T_{\pm3^{2\alpha+1}( 3^{2\alpha+2})}\left(3^{2\alpha+2\beta+2}p^{2k+1}n+\frac{p^{2k+2}\times3^{2\alpha+2\beta+2}+2\times3^{2\alpha+1}+1}{8}\right)\equiv 0 \pmod{3^{3\alpha+\beta+4}}, \label{MR11}
\end{align}
where $p$ is an odd prime and $p\nmid n$.
\end{theorem}

\begin{theorem}\label{MT4}For any integers $n, \alpha, \beta \geq 0$, we have
\begin{align}
\sum_{n=0}^\infty T_{\pm3^{2\alpha+2}( 3^{2\alpha+3})}\left(3^{2\alpha+2\beta+2}n+\frac{5\times3^{2\alpha+2\beta+2}+2\times3^{2\alpha+2}+1}{8}\right)\equiv 0 \pmod{3^{3\alpha+3\beta+4}},\label{MR12}\\
\sum_{n=0}^\infty T_{\pm3^{2\alpha+2}( 3^{2\alpha+3})}\left(3^{2\alpha+2\beta+3}n+\frac{13\times3^{2\alpha+2\beta+2}+2\times3^{2\alpha+2}+1}{8}\right)\equiv 0 \pmod{3^{3\alpha+3\beta+5}},\label{MR13}
\end{align}
and 
\begin{align}
\sum_{n=0}^\infty T_{\pm3^{2\alpha+2}( 3^{2\alpha+3})}\left(3^{2\alpha+2\beta+3}n+\frac{7\times3^{2\alpha+2\beta+3}+2\times3^{2\alpha+2}+1}{8}\right) \equiv 0 \pmod{3^{3\alpha+3\beta+6}}.\label{MR14}
\end{align}
\end{theorem}

\begin{theorem}\label{MT5}For any integers $n, \alpha, \beta, k \geq 0$. If $p$ is a prime such that $(-3/p)=-1$ and $p\nmid n$, then 
\begin{equation}\label{MR15}
p_{3^{2\alpha+1},3}\left(3^{2\alpha+\beta+1}n+\frac{4\times3^{2\alpha+\beta+1}+3^{2\alpha+1}+1}{8}\right )
\equiv 0 \pmod{3^{3\alpha+\beta+2}}
\end{equation}
and
\begin{equation}\label{MR16}
p_{3^{2\alpha+1},3}\left(3^{2\alpha+\beta+1}p^{2k+1}n+\frac{4p^{2k+2}\times3^{2\alpha+\beta+1}+3^{2\alpha+1}+1}{8}\right )
\equiv 0 \pmod{3^{3\alpha+\beta+4}}.
\end{equation}
\end{theorem}

\begin{theorem}\label{MT6}
For any integers $n, \alpha, \beta, k \geq 0$ and prime $p$ such that $p\equiv3\pmod{4}$ and $p\nmid n$, we have
\begin{align}
p_{3^{2\alpha+2},3}\left(3^{2\alpha+2\beta+1}n+\frac{2\times3^{2\alpha+2\beta+1}+3^{2\alpha+2}+1}{8}\right )&\equiv 0 \pmod{3^{3\alpha+2\beta+2}},\label{MR17}\\
p_{3^{2\alpha+2},3}\left(3^{2\alpha+2\beta+2}n+\frac{10\times3^{2\alpha+2\beta+1}+3^{2\alpha+2}+1}{8}\right )&\equiv 0 \pmod{3^{3\alpha+2\beta+3}},\label{MR171}\\
p_{3^{2\alpha+2},3}\left(3^{2\alpha+2\beta+3}+\frac{14\times3^{2\alpha+2\beta+2}+3^{2\alpha+2}+1}{8}\right )&\equiv 0 \pmod{3^{3\alpha+2\beta+6}},\label{MR18}\\
p_{3^{2\alpha+2},3}\left(3^{2\alpha+2\beta+3}n+\frac{22\times3^{2\alpha+2\beta+2}+3^{2\alpha+2}+1}{8}\right )&\equiv 0 \pmod{3^{3\alpha+2\beta+7}},\label{MR19}
\end{align}
and
\begin{equation}
p_{3^{2\alpha+2},3}\left(3^{2\alpha+2\beta+1}p^{2k+1}n+\frac{2p^{2k+2}\times3^{2\alpha+2\beta+1}+3^{2\alpha+2}+1}{8}\right )\equiv 0 \pmod{3^{3\alpha+2\beta+4}}.\label{MR20}
\end{equation}
\end{theorem}

\begin{theorem}\label{MT7}
\begin{align}
p_{\pm3^{2\alpha+1}(3^{2\alpha+2}),3}\left(3^{2\alpha+2\beta+1}n+\frac{7\times3^{2\alpha+2\beta+1}-2\times3^{2\alpha+1}+1}{8}\right) &\equiv 0 \pmod{3^{3\alpha+3\beta+2}},  \label{MR21} \\
p_{\pm3^{2\alpha+1}(3^{2\alpha+2}),3}\left(3^{2\alpha+2\beta+2}n+\frac{23\times3^{2\alpha+2\beta+1}-2\times3^{2\alpha+1}+1}{8}\right) &\equiv 0 \pmod{3^{3\alpha+3\beta+4}},  \label{MR22} \\
p_{\pm3^{2\alpha+1}(3^{2\alpha+2}),3}\left(3^{2\alpha+2\beta+2}n+\frac{5\times3^{2\alpha+2\beta+2}-2\times3^{2\alpha+1}+1}{8}\right) &\equiv 0 \pmod{3^{3\alpha+3\beta+3}},  \label{MR23}  
\end{align}
and 
\begin{equation}\label{MR24}
p_{\pm3^{2\alpha+1}(3^{2\alpha+2}),3}\left(3^{2\alpha+2\beta+3}n+\frac{13\times3^{2\alpha+2\beta+2}-2\times3^{2\alpha+1}+1}{8}\right) \equiv 0 \pmod{3^{3\alpha+3\beta+4}}.
\end{equation}
\end{theorem}
The remainder of the paper is structured as follows. In Section \ref{SS2}, we introduce the necessary background on the $H$ operator along with some useful lemmas. In Section \ref{SS3}, we provide proofs for Theorems \ref{MT1}-\ref{MT4}, and in Section \ref{SS4}, we prove Theorems \ref{MT5}-\ref{MT7}.

\section{Preliminaries} \label{SS2}
To establish the main results of this paper, we first introduce the necessary notation and lemmas. We first introduce a huffing operator $H$, which is defined as 
\begin{equation*}
H\left(\sum_{n=0}^\infty a_n q^n\right)=\sum_{n=0}^\infty a_{3n} q^{3n}.
\end{equation*}
\begin{lemma}\cite[(2.23)]{H}
Let $\zeta=\frac{E(q)^3}{q E(q^9)^3}$ and $T =\frac{E(q^3)^{12}}{q^3 E(q^9)^{12}}$. Then
\begin{equation}\label{P0}
H\left(\frac{1}{\zeta^i}\right) = \sum_{j\geq1} \frac{m_{i, j}}{T^j}.
\end{equation}
For a positive integer $n$, let $\pi(n)$ be the highest power of $3$ that divides $n$, and define $\pi(0)=+\infty$.
\begin{lemma}[\cite{H} and \cite{T}]
    For each $i,j\geq1 $, we have \begin{equation}\label{pmij} \pi\left(m_{i,j}\right)\geq\left\lfloor\frac{9j-3i-1}{2}\right\rfloor.
    \end{equation}
     \end{lemma}
\begin{proof}
It follows easily from \eqref{R1} and induction.
\end{proof}

\begin{lemma}\cite[(4.6) and (4.7)]{H}
For any integers $\alpha\geq0$ and $j \geq 1$, we have
\begin{equation}\label{C1}
\pi ({x_{2\alpha+1, j}})\geq 3\alpha+2+\left\lfloor\frac{9j-10}{2}\right \rfloor+\delta_{j,1}
\end{equation}
and
\begin{equation}\label{C2}
\pi ({x_{2\alpha+2, j}})\geq 3\alpha+4+\left\lfloor\frac{9j-10}{2}\right \rfloor+\delta_{j,1},
\end{equation}
where \begin{equation*}
   \delta_{j, 1}=\begin{cases}
   1 & \text{if $j=1$},\\
    0 & \text{if $j\neq 1$}.
    \end{cases}
\end{equation*}
\end{lemma}

\end{lemma}
\begin{lemma}
For all $i\geq 1$, we have
\begin{align}
    H\left(q^{i-3}\frac{E(q^3)^{12i}}{E(q)^{12i}}\right)&=\sum_{j\geq 1}m_{4i,i+j} \, q^{3j-3} \frac{E(q^9)^{12j}}{E(q^3)^{12j}},\label{P3} \\
    H\left(q^{i-2}\frac{E(q^3)^{12i-3}}{E(q)^{12i+3}}\right)&=\sum_{j\geq 1}m_{4i+1,i+j} \, q^{3j-3}\frac{E(q^9)^{12j-3}}{E(q^3)^{12j+3}},\label{P4} \\
    H\left(q^{i-3}\frac{E(q^3)^{12i-9}}{E(q)^{12i-9}}\right)&=\sum_{j\geq 1}m_{4i-3,i+j-1}\, q^{3j-3}\frac{E(q^9)^{12j-3}}{E(q^3)^{12j-3}},\label{P1} \\
    H\left(q^{i-1}\frac{E(q^3)^{12i-3}}{E(q)^{12i-3}}\right)&=\sum_{j\geq 1}m_{4i-1,i+j-1}\, q^{3j-3}\frac{E(q^9)^{12j-9}}{E(q^3)^{12j-9}},\label{P2} 
\end{align}
and
\begin{equation}\label{P5}
    H\left(q^{i-1}\frac{E(q^3)^{12i}}{E(q)^{12i+6}}\right)=\sum_{j\geq 1}m_{4i+2,i+j}\, q^{3j-3}\frac{E(q^9)^{12j-6}}{E(q^3)^{12j}}.
\end{equation}
\end{lemma}
\begin{proof}
We can rewrite \eqref{P0} as
\begin{equation}\label{P00}
H\left(q^{i}\frac{E(q^9)^{3i}}{E(q)^{3i}}\right)=\sum_{j= 1}^\infty m_{i,j} \, q^{3j} \frac{E(q^9)^{12j}}{E(q^3)^{12j}}.
\end{equation}
From \eqref{P00} and the fact that $m_{4i, j}=0$ for $1\leq j \leq i$, we see that
\begin{align*}
        H\left(q^{4i}\frac{E(q^9)^{12i}}{E(q)^{12i}}\right)&=\sum_{j=i+1}^\infty m_{4i,j} \, q^{3j} \frac{E(q^9)^{12j}}{E(q^3)^{12j}}\\
        &=\sum_{j=1}^\infty m_{4i,j+i} \, q^{3j+3i} \frac{E(q^9)^{12j+12i}}{E(q^3)^{12j+12i}}
\end{align*}
which yields \eqref{P3}. Similarly, we can prove the identities \eqref{P4}--\eqref{P5}.
\end{proof}
\section{Congruences modulo powers of 3 for $T_{3^{\ell}}(n)$}\label{SS3}
In this section, we derive generating functions for $T_{3^{\ell}}(n)$, which we use to prove Theorems \ref{MT1}-\ref{MT4}. 
\begin{theorem} For each $\alpha \geq0$ and $\beta \geq 0$, we have
\begin{align}
\sum_{n\ge0}T_{3^{2\alpha+1}}\left(3^{2\alpha+2\beta+1}n+\frac{2\cdot3^{2\alpha+2\beta+2}-3^{2\alpha+1}+1}{8}\right )q^n
=\sum_{j\geq 1}r^{(2\alpha+1)}_{2\beta+1,j}q^{j-1}\frac{E(q^3)^{12j-3}}{E(q)^{12j-3}},\label{T11}
\end{align}
and
\begin{align}
\sum_{n\ge0}T_{3^{2\alpha+1}}\left(3^{2\alpha+2\beta+2}n+\frac{2\cdot3^{2\alpha+2\beta+2}-3^{2\alpha+1}+1}{8}\right )q^{n}=\sum_{j\geq 1}r^{(2\alpha+1)}_{2\beta+2,j}q^{j-1}\frac{E(q^3)^{12j-9}}{E(q)^{12j-9}}.\label{D6}
\end{align}
where the coefficient vectors are defined as follows:
\begin{equation}\label{T111}
        r_{1,j}^{(2\alpha+1)}=x_{2\alpha+1,j}
\end{equation}
and
\begin{equation}\label{T112}
       r^{(2\alpha+1)}_{\mu+1,j}= 
       \begin{cases}
      \displaystyle \sum_{i\geq1}r^{(2\alpha+1)}_{\mu,i} m_{4i-1,i+j-1} \,\,\   & \text{if  $\mu$ is odd},\\
       \displaystyle \sum_{i\geq1}r^{(2\alpha+1)}_{\mu,i}m_{4i-3,i+j-1} \,\,\   & \text{if  $\mu$ is even},
       \end{cases}
\end{equation}
for all $\mu, j\geq 1$.
\end{theorem}
\begin{proof}
From \eqref{D2}, we have
\begin{equation}\label{D3}
\sum_{n\geq 0}T_{3^{2\alpha+1}}(n)q^n=E\left(q^{3^{2\alpha+1}}\right)^3\sum_{n\geq 0}p_{3}(n)q^n.
\end{equation}
Extracting the terms involving $q^{3^{2\alpha+1}n+\frac{5\times3^{2\alpha+1}+1}{8}}$ on both sides of \eqref{D3} and dividing throughout by $q^{\frac{5\times3^{2\alpha+1}+1}{8}}$, we obtain
\begin{align}
\sum_{n\geq 0}T_{3^{2\alpha+1}}\left(3^{2\alpha+1}n+\frac{5\times3^{2\alpha+1}+1}{8}\right)q^{3^{2\alpha+1}n}=E\left(q^{3^{2\alpha+1}}\right)^3\sum_{n\geq 0}p_{3}\left(3^{2\alpha+1}n+\frac{5\times3^{2\alpha+1}+1}{8}\right)q^{3^{2\alpha+1}n}.\label{D4}
\end{align}
If we replace $q^{3^{2\alpha+1}}$ by $q$ and  use \eqref{H1}, we get
\begin{equation}\label{D5}
\sum_{n\geq 0}T_{3^{2\alpha+1}}\left(3^{2\alpha+1}n+\frac{5\times3^{2\alpha+1}+1}{8}\right)q^n=\sum_{j\geq 1}x_{2\alpha+1,j}q^{j-1}\frac{E(q^3)^{12j-3}}{E(q)^{12j-3}},
\end{equation}
which is the case $\beta=0$ of \eqref{T11}.

We now assume that \eqref{T11} is true for some integer $\beta\geq 0$. Applying the operator $H$ to both sides, by \eqref{P2}, we have

\begin{align*}
&\sum_{n\ge0}T_{3^{2\alpha+1}}\left(3^{2\alpha+2\beta+2}n+\frac{2\cdot3^{2\alpha+2\beta+2}-3^{2\alpha+1}+1}{8}\right )q^{3n}\\&
=\sum_{i\geq 1}r^{(2\alpha+1)}_{2\beta+1,i}\sum_{j\geq 1}m_{4i-1,i+j-1}q^{3j-3}\left(\frac{E(q^9)^3}{E(q^3)^3}\right)^{4j-3}\\&
=\sum_{j\geq 1}\left(\sum_{i\geq 1}r^{(2\alpha+1)}_{2\beta+1,i}m_{4i-1,i+j-1}\right)q^{3j-3}\left(\frac{E(q^9)^3}{E(q^3)^3}\right)^{4j-3}\\&
=\sum_{j\geq 1}r^{(2\alpha+1)}_{2\beta+2,j}q^{3j-3}\left(\frac{E(q^9)^3}{E(q^3)^3}\right)^{4j-3}.
\end{align*}
If we now replace $q^3$ by $q$, we obtain
\begin{align*}
\sum_{n\ge0}T_{3^{2\alpha+1}}\left(3^{2\alpha+2\beta+2}n+\frac{2\cdot3^{2\alpha+2\beta+2}-3^{2\alpha+1}+1}{8}\right )q^{n-2}=\sum_{j\geq 1}r^{(2\alpha+1)}_{2\beta+2,j}q^{j-3}\left(\frac{E(q^3)^3}{E(q)^3}\right)^{4j-3},
\end{align*}
which is \eqref{D6}. 

Now suppose \eqref{D6} holds. If we operate $H$ on both sides of the equation, by \eqref{P1}, we have
\begin{align*}
&\sum_{n\ge0}T_{3^{2\alpha+1}}\left(3^{2\alpha+2\beta+3}n+\frac{2\cdot3^{2\alpha+2\beta+4}-3^{2\alpha+1}+1}{8}\right )q^{3n}\\&
=\sum_{j\geq 1}r^{(2\alpha+1)}_{2\beta+2,j}\sum_{j\geq 1}m_{4i-3,i+j-1}q^{3j-3}\left(\frac{E(q^9)^3}{E(q^3)^3}\right)^{4j-1}\\&
=\sum_{j\geq 1}\left(\sum_{i\geq 1}r^{(2\alpha+1)}_{2\beta+2,i}m_{4i-3,i+j-1}\right)q^{3j-3}\left(\frac{E(q^9)^3}{E(q^3)^3}\right)^{4j-1}\\&
=\sum_{j\geq 1}r^{(2\alpha+1)}_{2\beta+3,j}q^{3j-3}\left(\frac{E(q^9)^3}{E(q^3)^3}\right)^{4j-1}.
\end{align*}
That is,
\begin{align*}
\sum_{n\ge0}T_{3^{2\alpha+1}}\left(3^{2\alpha+2\beta+3}n+\frac{2\cdot3^{2\alpha+2\beta+4}-3^{2\alpha+1}+1}{8}\right )q^n
=\sum_{j\geq 1}r^{(2\alpha+1)}_{2\beta+3,i}q^{j-1}\left(\frac{E(q^3)^3}{E(q)^3}\right)^{4j-1}.
\end{align*}
So we obtain \eqref{T11} with $\beta$ replaced by $\beta+1$.
\end{proof}

\begin{lemma}
For any integers $\alpha, \beta\geq0$ and $j \geq 1$, we have
\begin{equation}\label{C3}
\pi ({r_{2\beta+1, j}^{(2\alpha+1)}})\geq 3\alpha+2\beta+2+\left\lfloor\frac{9j-10}{2}\right \rfloor+\delta_{j,1}
\end{equation}
and
\begin{equation}\label{C4}
\pi ({r_{2\beta+2, j}^{(2\alpha+1)}})\geq 3\alpha+2\beta+2+\left\lfloor\frac{9j-10}{2}\right \rfloor+\delta_{j,1}.
\end{equation}
\end{lemma}
\begin{proof}
By \eqref{T111} and \eqref{C1}, \eqref{C3} is true for $\beta=0$.
Now suppose \eqref{C3} is true for some $\beta\geq0$. Then
\begin{align*}
\pi ({r_{2\beta+2, j}^{(2\alpha+1)}})&\geq \min_{i\geq1} \left(\pi ({r_{2\beta+1, i}^{(2\alpha+1)}})+\pi(m_{4i-1, i+j-1})\right)\\&
= \pi(r_{2\beta+1, 1}^{(2\alpha+1)}) + \pi\left( m(3, j)\right)\\&
\geq 3\alpha+2\beta+2 + \left\lfloor\frac{9j-10}{2}\right \rfloor + \delta_{j,1},
\end{align*}
which is \eqref{C4}. 

Now suppose \eqref{C4} is true. Then
\begin{align*}
\pi ({r_{2\beta+3, j}^{(2\alpha+1)}})&\geq \min_{i\geq1} \left(\pi ({r_{2\beta+2, i}^{(2\alpha+1)}})+\pi(m_{4i-3, i+j-1})\right)\\&
\geq 3\alpha+2\beta+2 + \left\lfloor\frac{9j-4}{2}\right \rfloor\\&
\geq 3\alpha+2\beta+4 + \left\lfloor\frac{9j-10}{2}\right \rfloor + \delta_{j,1},
\end{align*}
which is \eqref{C3} with $\beta+1$ for $\beta$. This completes the proof by induction.
\end{proof}

\begin{proof}[Proof of Theorem \ref{MT1}]
From \eqref{T11} and \eqref{C3}, we have
\begin{align}
\sum_{n\ge0}T_{3^{2\alpha+1}}\left(3^{2\alpha+2\beta+1}n+\frac{2\cdot3^{2\alpha+2\beta+2}-3^{2\alpha+1}+1}{8}\right )q^n
\equiv r^{(2\alpha+1)}_{2\beta+1,1}\frac{E(q^3)^9}{E(q)^9}    \pmod{3^{3\alpha+2\beta+6}}.\label{PP1}
\end{align}
This yields the family of congruences \eqref{MR1}. 

Due to Jacobi, we have
\begin{equation}\label{PP7}
    E(q)^3 = \sum_{n=0}^\infty (-1)^n (2n+1) q^{n(n+1)/2}. 
\end{equation}
Which is equivalent to
\begin{equation}\label{PP2}
    E(q)^3 = P(q^3) - 3q E(q^9)^3,
\end{equation}
where \[P(q) = \sum_{n=-\infty}^\infty (-1)^n (6n+1) q^{n(3n+1)/2}. \]
By the binomial theorem, it is easy that
\[\frac{E(q^3)^9}{E(q)^9} \equiv E(q)^{18} \pmod{3^3}\]
and by \eqref{PP2}
\[E(q)^{18} \equiv P(q^3)^6-18q P(q^3)^5E(q^9)^3  \pmod{3^3}.\]
With the aid of above, \eqref{PP1} reduces to
\begin{align}
\nonumber&\sum_{n\ge0}T_{3^{2\alpha+1}}\left(3^{2\alpha+2\beta+1}n+\frac{2\cdot3^{2\alpha+2\beta+2}-3^{2\alpha+1}+1}{8}\right )q^n\\&
\equiv r^{(2\alpha+1)}_{2\beta+1,1}(P(q^3)^6-18q P(q^3)^5E(q^9)^3 )    \pmod{3^{3\alpha+2\beta+5}}.\label{PP3}
\end{align}
If we extract the terms involving $q^{3n+1}$ and  $q^{3n+2}$ from both sides of \eqref{PP3}, we arrive at \eqref{MR2} and \eqref{MR3}, respectively.

By using the fact, \[\frac{E(q^3)^9}{E(q)^9} \equiv E(q^3)^{6} \pmod{3^2}\] in \eqref{PP1} and extracting the terms involving $q^{3n}$ from both sides, we find that
\begin{align*}
&\sum_{n\ge0}T_{3^{2\alpha+1}}\left(3^{2\alpha+2\beta+2}n+\frac{2\cdot3^{2\alpha+2\beta+2}-3^{2\alpha+1}+1}{8}\right )q^n\\&
\equiv r^{(2\alpha+1)}_{2\beta+1,1} E(q)^6 \\&
\equiv r^{(2\alpha+1)}_{2\beta+1,1} \sum_{k=0}^\infty \sum_{l=0}^\infty (-1)^{k+l} (2k+1)(2l+1) q^{k(k+1)/2+l(l+1)/2} \pmod{3^{3\alpha+2\beta+4}}
\end{align*}
which implies that
\begin{align*}
&\sum_{n\ge0}T_{3^{2\alpha+1}}\left(3^{2\alpha+2\beta+2}n+\frac{2\cdot3^{2\alpha+2\beta+2}-3^{2\alpha+1}+1}{8}\right )q^{8n+2}\\&
\equiv r^{(2\alpha+1)}_{2\beta+1,1} \sum_{k=0}^\infty \sum_{l=0}^\infty (-1)^{k+l} (2k+1)(2l+1) q^{(2k+1)^2 +(2l+1)^2} \pmod{3^{3\alpha+2\beta+4}}
\end{align*}
Thus, $T_{3^{2\alpha+1}}\left(3^{2\alpha+2\beta+2}n+\frac{2\cdot3^{2\alpha+2\beta+2}-3^{2\alpha+1}+1}{8}\right ) \equiv 0 \pmod{3^{3\alpha+2\beta+4}}$ if $8n+2$ is not a sum of two squares.
We recall that an integer $N>0$ cannot be written as a sum of two squares if and only if each prime factor $p \equiv 3 \pmod{4}$ has an odd power in the prime decomposition of $N$. However, we have 
\[n=p^{2k+1} m +\frac{p^{2k+2}-1}{4}\]
which yields $8n+2 = 8 p^{2k+1}m+2p^{2k+2} = 2p^{2k+1}(4m+p).$ If $p\nmid m$, then $8n+2$ is not a sum of two squares. This completes the proof of \eqref{MR4}.
\end{proof}

\begin{theorem} For each $\alpha, \beta \geq 0$, we have
\begin{equation}\label{T12}
\sum_{n\ge0}T_{3^{2\alpha+2}}\left(3^{2\alpha+\beta+1}(n+1)-\frac{3^{2\alpha+2}-1}{8}\right )q^n
=\sum_{j\geq 1}s^{(2\alpha+2)}_{\beta+1,j}q^{j-1}\frac{E(q^3)^{12j}}{E(q)^{12j}},
\end{equation}
where the coefficient vectors are defined as follows:
\begin{equation}\label{T121}
        s_{1,j}^{(2\alpha+2)}=x_{2\alpha+1,j}
\end{equation}
and
\begin{equation}\label{T122}
s^{(2\alpha+2)}_{\mu+1,j}=\sum_{i\geq1}s^{(2\alpha+2)}_{\mu,i} m_{4i,i+j}
\end{equation}
for all $\mu, j\geq 1$.
\end{theorem}
\begin{proof}
From \eqref{D2}, we have
\begin{equation}\label{D7}
\sum_{n\geq 0}T_{3^{2\alpha+2}}(n)q^n=E\left(q^{3^{2\alpha+2}}\right)^3\sum_{n\geq 0}p_{3}(n)q^n.
\end{equation}
Extracting the terms involving $q^{3^{2\alpha+1}n+\frac{5\times3^{2\alpha+1}+1}{8}}$ on both sides of \eqref{D7} and dividing throughout by $q^{\frac{5\times3^{2\alpha+1}+1}{8}}$, we obtain
\begin{align}
\sum_{n\geq 0}T_{3^{2\alpha+2}}\left(3^{2\alpha+1}n+\frac{5\times3^{2\alpha+1}+1}{8}\right)q^{3^{2\alpha+1}n}=E\left(q^{3^{2\alpha+2}}\right)^3\sum_{n\geq 0}p_{3}\left(3^{2\alpha+1}n+\frac{5\times3^{2\alpha+1}+1}{8}\right)q^{3^{2\alpha+1}n}.\label{D8}
\end{align}
If we replace $q^{3^{2\alpha+1}}$ by $q$ and use \eqref{H1}, we get
\begin{equation}\label{D9}
\sum_{n\geq 0}T_{3^{2\alpha+2}}\left(3^{2\alpha+1}n+\frac{5\times3^{2\alpha+1}+1}{8}\right)q^n=\sum_{j\geq 1}x_{2\alpha+1, j}\,q^{j-1}\frac{E(q^3)^{12j}}{E(q)^{12j}},
\end{equation}
which is the case $\beta=0$ of \eqref{T12}. By \eqref{P3} and induction on $\beta$, we can easily prove that equation \eqref{T12} is true for all $\beta \geq1$.
\end{proof}

\begin{lemma}
For any integers $\alpha, \beta\geq0$ and $j \geq 1$, we have
\begin{equation}\label{C5}
\pi ({s_{\beta+1, j}^{(2\alpha+2)}})\geq 3\alpha+2\beta+2+\left\lfloor\frac{9j-10}{2}\right \rfloor+\delta_{j,1}
\end{equation}
\end{lemma}
\begin{proof}
By \eqref{T121} and \eqref{C1}, \eqref{C5} is true for $\beta=0$.
Now suppose \eqref{C5} is true for some $\beta\geq0$. Then
\begin{align*}
\pi ({s_{\beta+2, j}^{(2\alpha+2)}})&\geq \min_{i\geq1} \left(\pi ({s_{\beta+1, i}^{(2\alpha+2)}})+\pi(m_{4i, i+j})\right)\\&
= \pi(s_{\beta+1, 1}^{(2\alpha+2)}) + \pi\left( m(4, 1+j)\right)\\&
\geq 3\alpha+2\beta+2 + \left\lfloor\frac{9j-4}{2}\right \rfloor \\&
\geq 3\alpha+2\beta+4 + \left\lfloor\frac{9j-10}{2}\right \rfloor + \delta_{j,1},
\end{align*}
which is \eqref{C5} with $\beta+1$ for $\beta$.
\end{proof}

\begin{proof}[Proof of Theorem \ref{MT2}]
From \eqref{T12} and \eqref{C5}, we have
\begin{align}
\nonumber \sum_{n\ge0}T_{3^{2\alpha+2}}\left(3^{2\alpha+\beta+1}(n+1)-\frac{3^{2\alpha+2}-1}{8}\right )q^n
&\equiv s^{(2\alpha+2)}_{\beta+1,j}\frac{E(q^3)^{12}}{E(q)^{12}}\\
& \equiv s^{(2\alpha+2)}_{\beta+1,j} E(q^3)^8 \pmod{3^{3\alpha+2\beta+3}}. \label{PP4}
\end{align}
Congruences \eqref{MR5} and \eqref{MR6} are evident from \eqref{PP4}.
\end{proof}
The proofs of the following results follow the same arguments as those used for earlier results in this section, and thus we omit the details.
\begin{theorem}
For any integers $\alpha, \beta \geq 0$, we have
\begin{align}
\sum_{n=0}^\infty T_{\pm3^{2\alpha+1}( 3^{2\alpha+2})}\left(3^{2\alpha+2\beta+1}n+\frac{3^{2\alpha+2\beta+2}+2\times3^{2\alpha+1}+1}{8}\right)q^n = \sum_{k\geq1} y_{2\beta+1, k}^{(2\alpha+1)}\, q^{k-1} \frac{E(q^3)^{12k-6}}{E(q)^{12k-3}} \label{T31}
\end{align}
and
\begin{align}
\sum_{n=0}^\infty T_{\pm3^{2\alpha+1}( 3^{2\alpha+2})}\left(3^{2\alpha+2\beta+2}n+\frac{3^{2\alpha+2\beta+2}+2\times3^{2\alpha+1}+1}{8}\right)q^n = \sum_{k\geq1} y_{2\beta+2, k}^{(2\alpha+1)}\, q^{k-1} \frac{E(q^3)^{12k-9}}{E(q)^{12k-6}}. \label{T311}
\end{align}
where the coefficient vectors are defined as follows:
\begin{equation*}
        y_{1,k}^{(2\alpha+1)}=x_{2\alpha+1,k}
\end{equation*}
and
\begin{equation*}
       y^{(2\alpha+1)}_{\mu+1,k}= 
       \begin{cases}
      \displaystyle \sum_{j\geq1}y^{(2\alpha+1)}_{\mu,j} m_{4j-1, j+k-1}   & \text{if  $\mu$ is odd},\\
       \displaystyle \sum_{j\geq1}y^{(2\alpha+1)}_{\mu,j}m_{4j-2,j+k-1}    & \text{if  $\mu$ is even},
       \end{cases}
\end{equation*}
for all $\mu, j\geq 1$.
\end{theorem}
\begin{lemma}
For any integers $\alpha, \beta\geq0$ and $j \geq 1$, we have
\begin{equation}\label{C6}
\pi ({y_{2\beta+1, j}^{(2\alpha+1)}})\geq 3\alpha+\beta+2+\left\lfloor\frac{9j-10}{2}\right \rfloor+\delta_{j,1}
\end{equation}
and
\begin{equation}\label{C7}
\pi ({y_{2\beta+2, j}^{(2\alpha+1)}})\geq 3\alpha+\beta+2+\left\lfloor\frac{9j-10}{2}\right \rfloor+\delta_{j,1}.
\end{equation}
\end{lemma}

\begin{theorem}\label{RT1}
For any integers $\alpha, \beta \geq 0$, we have
\begin{align}
\nonumber &\sum_{n=0}^\infty T_{\pm3^{2\alpha+2}( 3^{2\alpha+3})}\left(3^{2\alpha+2\beta+2}n+\frac{5\times3^{2\alpha+2\beta+2}+2\times3^{2\alpha+2}+1}{8}\right)q^n \\&= \sum_{k\geq1} z_{2\beta+1, k}^{(2\alpha+2)}\, q^{k-1} \frac{E(q^3)^{12k-3}}{E(q)^{12k}}\label{T32}
\end{align}
and
\begin{align}
\nonumber &\sum_{n=0}^\infty T_{\pm3^{2\alpha+2}( 3^{2\alpha+3})}\left(3^{2\alpha+2\beta+3}n+\frac{7\times3^{2\alpha+2\beta+3}+2\times3^{2\alpha+2}+1}{8}\right)q^n \\&= \sum_{k\geq1} z_{2\beta+2, k}^{(2\alpha+2)}\, q^{k-1} \frac{E(q^3)^{12k}}{E(q)^{12k+3}}. \label{T321}
\end{align}
where the coefficient vectors are defined as follows:
\begin{equation*}
        z_{1,k}^{(2\alpha+2)}=x_{2\alpha+2,k}
\end{equation*}
and
\begin{equation*}
       z^{(2\alpha+2)}_{\mu+1,k}= 
       \begin{cases}
      \displaystyle \sum_{j\geq1}z^{(2\alpha+2)}_{\mu,j} m_{4j, j+k}   & \text{if  $\mu$ is odd},\\
       \displaystyle \sum_{j\geq1}z^{(2\alpha+2)}_{\mu,j}m_{4j+1,j+k}    & \text{if  $\mu$ is even},
       \end{cases}
\end{equation*}
for all $\mu, j\geq 1$.
\end{theorem}

\begin{lemma}\label{RL1}
For any integers $\alpha, \beta\geq0$ and $j \geq 1$, we have
\begin{equation}\label{C8}
\pi ({z_{2\beta+1, j}^{(2\alpha+2)}})\geq 3\alpha+3\beta+4+\left\lfloor\frac{9j-10}{2}\right \rfloor+\delta_{j,1}
\end{equation}
and
\begin{equation}\label{C9}
\pi ({z_{2\beta+2, j}^{(2\alpha+2)}})\geq 3\alpha+3\beta+6+\left\lfloor\frac{9j-10}{2}\right \rfloor+\delta_{j,1}.
\end{equation}
\end{lemma}

\begin{proof}[Proof of Theorem \ref{MT3}]
Congruence \eqref{MR7} follows directly from \eqref{T31} and \eqref{C6}. Using the binomial theorem and \eqref{C6} in \eqref{T31}, we obtain
\begin{align}
\nonumber &\sum_{n=0}^\infty T_{\pm3^{2\alpha+1}( 3^{2\alpha+2})}\left(3^{2\alpha+2\beta+1}n+\frac{3^{2\alpha+2\beta+2}+2\times3^{2\alpha+1}+1}{8}\right)q^n \\&  \nonumber \equiv y_{2\beta+1, 1}^{(2\alpha+1)}\frac{E(q)^{18}}{E(q^3)^3} \\& \equiv y_{2\beta+1, 1}^{(2\alpha+1)} (P(q^3)^6-18q P(q^3)^5E(q^9)^3) \pmod{3^{3\alpha+\beta+5}}.\label{PP5}
\end{align}
By extracting the terms involving $q^{3n+1}$ and  $q^{3n+2}$ from both sides of \eqref{PP5} , we arrive at \eqref{MR8} and \eqref{MR9}, respectively.

In view of \eqref{T311} and \eqref{C7}, we arrive at
\begin{align}
\nonumber &\sum_{n=0}^\infty T_{\pm3^{2\alpha+1}( 3^{2\alpha+2})}\left(3^{2\alpha+2\beta+2}n+\frac{3^{2\alpha+2\beta+2}+2\times3^{2\alpha+1}+1}{8}\right)q^n \\& \equiv  y_{2\beta+2, 1}^{(2\alpha+1)} \frac{E(q^3)^3}{E(q)^6} \equiv  y_{2\beta+2, 1}^{(2\alpha+1)} E(q)^3 \pmod{3^{3\alpha+\beta+4}}. \label{PP6}
\end{align}
Congruence \eqref{MR10} follows directly from \eqref{PP6} and \eqref{PP7}. An immediate consequence of this is \eqref{MR11}.
\end{proof}

\begin{proof}[Proof of Theorem \ref{MT4}]
    Congruences \eqref{MR12} and \eqref{MR14} follow directly from Theorem \ref{RT1} and Lemma \ref{RL1}.

    In view of \eqref{T32} and \eqref{C8}, we arrive at
    \begin{align}
\nonumber &\sum_{n=0}^\infty T_{\pm3^{2\alpha+2}( 3^{2\alpha+3})}\left(3^{2\alpha+2\beta+2}n+\frac{5\times3^{2\alpha+2\beta+2}+2\times3^{2\alpha+2}+1}{8}\right)q^n \\&
\equiv z_{2\beta+1, 1}^{(2\alpha+2)} \frac{E(q^3)^{9}}{E(q)^{12}}\equiv E(q^3)^5 \pmod{3^{3\alpha +3\beta+5}}\label{PP8}
\end{align}
By extracting the terms involving $q^{3n+1}$ from both sides of \eqref{PP8} , we arrive at \eqref{MR13}.
\end{proof}

\section{Congruences modulo powers of 3 for $p_{3^{\ell},3}(n)$}\label{SS4}
\noindent  In this section, we derive the generating function for $p_{3^{\ell},3}(n)$ within certain arithmetic progressions, which are then used to prove our main theorems \ref{MT5}-\ref{MT7}. Since the proofs of the following results are similar to the results proved in the previous section, we skip the details of the proofs.
\begin{theorem} 
For each $\alpha \geq0$ and $\beta \geq 0$, we have
\begin{align}
\sum_{n\ge0}p_{3^{2\alpha+1},3}\left(3^{2\alpha+\beta+1}n+\frac{4\times3^{2\alpha+\beta+1}+3^{2\alpha+1}+1}{8}\right )q^n
=\sum_{j\geq 1}u^{(2\alpha+1)}_{\beta+1,j}q^{j-1}\left(\frac{E(q^3)^{12j-3}}{E(q)^{12j+3}}\right),\label{T21}
\end{align}
where the coefficient vectors are defined as follows:
\begin{equation*}
        u_{1,j}^{(2\alpha+1)}=x_{2\alpha+1,j}
\end{equation*}
and
\[u^{(2\alpha+1)}_{\mu+1,j}=\sum_{i\geq1}u^{(2\alpha+1)}_{\mu,i} m_{4i+1,i+j}\]
for all $\mu, j\geq 1$.
\end{theorem}

\begin{theorem} 
For each $\alpha, \beta \geq0$, we have
\begin{align}
\sum_{n\ge0}p_{3^{2\alpha+2},3}\left(3^{2\alpha+2\beta+1}n+\frac{2\times3^{2\alpha+2\beta+1}+3^{2\alpha+2}+1}{8}\right )q^n
=\sum_{j\geq 1}v^{(2\alpha+2)}_{2\beta+1,j}q^{j-1}\left(\frac{E(q^3)^{12j-6}}{E(q)^{12j}}\right)\label{T22}
\end{align}
and
\begin{align}
\sum_{n\ge0}p_{3^{2\alpha+2},3}\left(3^{2\alpha+2\beta+2}n+\frac{2\times3^{2\alpha+2\beta+3}+3^{2\alpha+2}+1}{8}\right )q^n
=\sum_{j\geq 1}v^{(2\alpha+2)}_{2\beta+2,j}q^{j-1}\left(\frac{E(q^3)^{12j}}{E(q)^{12j+6}}\right)\label{T23}
\end{align}
where the coefficient vectors are defined as follows:
\begin{equation*}
        v_{1,j}^{(2\alpha+2)}=x_{2\alpha+1,j}
\end{equation*}
and
\begin{equation*}
       v^{(2\alpha+2)}_{\mu+1,j}= 
       \begin{cases}
      \displaystyle \sum_{i\geq1}v^{(2\alpha+2)}_{\mu,i} m_{4i,i+j} & \text{if  $\mu$ is odd},\\
       \displaystyle \sum_{i\geq1}v^{(2\alpha+2)}_{\mu,i}m_{4i+2,i+j} & \text{if  $\mu$ is even},
\end{cases}
\end{equation*}
for all $\mu, j\geq 1$.
\end{theorem}

\begin{theorem} 
For each $\alpha, \beta \geq0$, we have
\begin{align}
\nonumber&\sum_{n\ge0}p_{\pm3^{2\alpha+1}(3^{2\alpha+2}),3}\left(3^{2\alpha+2\beta+1}n+\frac{7\times3^{2\alpha+2\beta+1}-2\times3^{2\alpha+1}+1}{8}\right )q^n\\&
=\sum_{j\geq 1}w^{(2\alpha+1)}_{2\beta+1,j}q^{j-1}\left(\frac{E(q^3)^{12j}}{E(q)^{12j+3}}\right)\label{T24}
\end{align}
and
\begin{align}
\nonumber&\sum_{n\ge0}p_{\pm3^{2\alpha+1}(3^{2\alpha+2}),3}\left(3^{2\alpha+2\beta+2}n+\frac{5\times3^{2\alpha+2\beta+2}-2\times3^{2\alpha+1}+1}{8}\right )q^n\\&
=\sum_{j\geq 1}w^{(2\alpha+1)}_{2\beta+2,j}q^{j-1}\left(\frac{E(q^3)^{12j-3}}{E(q)^{12j}}\right)\label{T25}
\end{align}
where the coefficient vectors are defined as follows:
\begin{equation*}
        w_{1,j}^{(2\alpha+1)}=x_{2\alpha+1,j}
\end{equation*}
and
\begin{equation*}
       w^{(2\alpha+1)}_{\mu+1,j}= 
       \begin{cases}
      \displaystyle \sum_{i\geq1}w^{(2\alpha+1)}_{\mu,i} m_{4i+1,i+j} & \text{if  $\mu$ is odd},\\
       \displaystyle \sum_{i\geq1}w^{(2\alpha+1)}_{\mu,i}m_{4i,i+j} & \text{if  $\mu$ is even},
\end{cases}
\end{equation*}
for all $\mu, j\geq 1$.
\end{theorem}

\begin{lemma}
For any integers $\alpha, \beta\geq0$ and $j \geq 1$, we have
\begin{equation}\label{C10}
\pi ({u_{\beta+1, j}^{(2\alpha+1)}})\geq 3\alpha+\beta+2+\left\lfloor\frac{9j-10}{2}\right \rfloor+\delta_{j,1}
\end{equation}
\end{lemma}

\begin{lemma}
For any integers $\alpha, \beta\geq0$ and $j \geq 1$, we have
\begin{equation}\label{C11}
\pi ({v_{2\beta+1, j}^{(2\alpha+2)}})\geq 3\alpha+2\beta+2+\left\lfloor\frac{9j-10}{2}\right \rfloor+\delta_{j,1}
\end{equation}
and
\begin{equation}\label{C12}
\pi ({v_{2\beta+2, j}^{(2\alpha+2)}})\geq 3\alpha+2\beta+4+\left\lfloor\frac{9j-10}{2}\right \rfloor+\delta_{j,1}.
\end{equation}
\end{lemma}

\begin{lemma}
For any integers $\alpha, \beta\geq0$ and $j \geq 1$, we have
\begin{equation}\label{C13}
\pi ({w_{2\beta+1, j}^{(2\alpha+1)}})\geq 3\alpha+3\beta+2+\left\lfloor\frac{9j-10}{2}\right \rfloor+\delta_{j,1}
\end{equation}
and
\begin{equation}\label{C14}
\pi ({w_{2\beta+2, j}^{(2\alpha+1)}})\geq 3\alpha+3\beta+3+\left\lfloor\frac{9j-10}{2}\right \rfloor+\delta_{j,1}.
\end{equation}
\end{lemma}

\begin{proof}[Proof of Theorem \ref{MT5}]
Congruence \eqref{MR15} is immediate from \eqref{T21} and \eqref{C10}. In view of \eqref{T21} and \eqref{C10}, we see that

\begin{align*}
&\sum_{n\ge0}p_{3^{2\alpha+1},3}\left(3^{2\alpha+\beta+1}n+\frac{4\times3^{2\alpha+\beta+1}+3^{2\alpha+1}+1}{8}\right )q^n\\&
\equiv u_{\beta+1, j}^{2\alpha+1}E(q)^3E(q^3)^3 \pmod{3^{3\alpha+\beta+4}}\\&
\equiv u^{(2\alpha+1)}_{\beta+1,1} \sum_{k=0}^\infty \sum_{l=0}^\infty (-1)^{k+l} (2k+1)(2l+1) q^{k(k+1)/2+3l(l+1)/2} \pmod{3^{3\alpha+2\beta+4}},
\end{align*}
which yields
\begin{align*}
&\sum_{n\ge0}p_{3^{2\alpha+1},3}\left(3^{2\alpha+\beta+1}n+\frac{4\times3^{2\alpha+\beta+1}+3^{2\alpha+1}+1}{8}\right )q^{8n+4}\\&
\equiv u^{(2\alpha+1)}_{\beta+1,1} \sum_{k=0}^\infty \sum_{l=0}^\infty (-1)^{k+l} (2k+1)(2l+1) q^{(2k+1)^2+3(2l+1)^2} \pmod{3^{3\alpha+2\beta+4}}.
\end{align*}
Thus, $p_{3^{2\alpha+1},3}\left(3^{2\alpha+\beta+1}n+\frac{4\times3^{2\alpha+\beta+1}+3^{2\alpha+1}+1}{8}\right )\equiv 0 \pmod{3^{3\alpha+2\beta+4}}$ if $8n+4$ is not of the form $(2k+1)^2+3(2l+1)^2$. We recall that an integer $N>0$ is not of the form $x^2 +3y^2$ if and only if for each prime $p$ such that $\left(\frac{-3}{p}\right)=-1$ has an odd power in the prime factorization of $N$. Let 
\begin{align*}
    N = 8\left( p^{2k+1}n + \frac{p^{2k+2}-1}{2}\right)+4=8 p^{2k+1}n + 4p^{2k+2}.
\end{align*}
Since $p\nmid n$ it follows that $8n+4$ is not of the form $x^2+3y^2$. Therefore, \eqref{MR16} holds.
\end{proof}
\begin{proof}[Proof of Theorem \ref{MT6}]
From \eqref{T22} and \eqref{C11}, we see that congruence \eqref{MR17} holds and we have 
\begin{align}
\nonumber &\sum_{n\ge0}p_{3^{2\alpha+2},3}\left(3^{2\alpha+2\beta+1}n+\frac{2\times3^{2\alpha+2\beta+1}+3^{2\alpha+2}+1}{8}\right )q^n\\&
\equiv v^{(2\alpha+2)}_{2\beta+1,1}E(q)^{6} \pmod{3^{3\alpha+2\beta+4}}\label{S3}\\&
\equiv v^{(2\alpha+2)}_{2\beta+1,1} (P(q^3)^2+9q^2E(q^9)^6-6qP(q^3)E(q^9)^3 \pmod{3^{3\alpha+2\beta+4}}. \label{S2}
\end{align}
By extracting the terms involving $q^{3n+1}$ from both sides, we arrive at congruences \eqref{MR171}.

In view of \eqref{T23} and \eqref{C12}, we have
\begin{align}
\nonumber&\sum_{n\ge0}p_{3^{2\alpha+2},3}\left(3^{2\alpha+2\beta+2}n+\frac{2\times3^{2\alpha+2\beta+3}+3^{2\alpha+2}+1}{8}\right )q^n\\&
\nonumber \equiv v^{(2\alpha+2)}_{2\beta+2,1} E(q^3)^3 E(q)^9 \pmod{3^{3\alpha+2\beta+7}}\\&
 \equiv v^{(2\alpha+2)}_{2\beta+2,1} E(q^3)^3 (P(q^3)^3-3^3q^3E(q^9)^9-3^2 q P(q^3)^2E(q^9)^3 \pmod{3^{3\alpha+2\beta+7}}.\label{S1}
\end{align}
By extracting the terms involving $q^{3n+1}$ and $q^{3n+2}$ from both sides of \eqref{S1}, we obtain congruences \eqref{MR18} and \eqref{MR19}.
Congruence \eqref{MR20} follows from \eqref{S3}, and the proof is similar to that of \eqref{MR4}; therefore, omit the details.
\end{proof}
\begin{proof}[Proof of Theorem \ref{MT7}]
From \eqref{T24} and \eqref{C13}, it follows that congruence \eqref{MR21} holds, and we have
\begin{align*}
&\sum_{n\ge0}p_{\pm3^{2\alpha+1}(3^{2\alpha+2}),3}\left(3^{2\alpha+2\beta+1}n+\frac{7\times3^{2\alpha+2\beta+1}-2\times3^{2\alpha+1}+1}{8}\right )q^n\\&
\equiv w^{(2\alpha+1)}_{2\beta+1,1}\frac{E(q^3)^{12}}{E(q)^{15}} \equiv w^{(2\alpha+1)}_{2\beta+1,1} E(q^3)^6 E(q)^3 \pmod{3^{3\alpha+3\beta+4}}
\end{align*}
By using \eqref{PP2} and extracting the terms involving $q^{3n+2}$ and $q^{3n+1}$ from the above, we obtain congruences \eqref{MR22} and \eqref{MR23}, respectively.

By \eqref{T25} and \eqref{C14}, we see that
\begin{align*}
&\sum_{n\ge0}p_{\pm3^{2\alpha+1}(3^{2\alpha+2}),3}\left(3^{2\alpha+2\beta+2}n+\frac{5\times3^{2\alpha+2\beta+2}-2\times3^{2\alpha+1}+1}{8}\right )q^n\\&
\equiv w^{(2\alpha+1)}_{2\beta+2,1}\frac{E(q^3)^{9}}{E(q)^{12}} 
\equiv w^{(2\alpha+1)}_{2\beta+2,1} E(q^3)^3 E(q)^6 \pmod{3^{3\alpha+3\beta+5}}
\end{align*}
By using \eqref{PP2} and extracting the terms involving $q^{3n+1}$ from the above, we obtain congruence \eqref{MR24}.
\end{proof}

\subsection*{Funding} No funding available for this article.
\subsection*{Data Availability} Data sharing is not applicable to this article as no data sets were generated or analyzed.

\end{document}